\documentclass[12pt]{article}
\usepackage{fullpage}
\usepackage{amsfonts, amsmath, amssymb,latexsym,epsfig}
\usepackage{amsthm}
\usepackage{tikz}
\usetikzlibrary{graphs}
\usetikzlibrary{graphs.standard}
\usepackage{relsize}
\usepackage{verbatim}
\usepackage{enumerate}
\usepackage[skip=10pt plus1pt, indent=20pt]{parskip}

\usepackage{hyperref}

\newtheorem{thm}{Theorem}[section]
\newtheorem{cor}[thm]{Corollary}
\newtheorem{lem}[thm]{Lemma}

\newtheorem{defn}[thm]{Definition}

\newcommand{\F}{\mathbb{F}}

\newcommand{\PG}{\text{PG}}
\newcommand{\AG}{\text{AG}}

\newcommand{\hs}{\hspace{1em}}

\newcommand{\mpr}[3]{m_{\text{Proj}}(#1 ; #2, #3)}
\newcommand{\mpa}[3]{m_{\text{Aff}}(#1 ; #2, #3)}

\usepackage[backend=biber]{biblatex} 
\DeclareFieldFormat{pages}{#1}
\author{
Jakob F\"uhrer\thanks{Institute for Algebra, Johannes Kepler University, Linz, Austria. E-mail: {\tt jakob.fuehrer@jku.at}.}
\and
Vladislav Taranchuk \thanks{Department of Mathematics: Analysis, Logic and Discrete Mathematics, Ghent University, 9000 Ghent, Belgium. E-mail: {\tt vlad.taranchuk@ugent.be}.}
}
\title{Large line-free sets and their applications}

\begin{document}

\date{\today}
\maketitle

\begin{abstract}
In this paper, we give an explicit construction of polynomials \( P \in \mathbb{F}_q[x_1,\dots,x_n] \) whose zero sets are large subsets of $\F_q^n$ which have restricted intersections with affine lines. We use these sets to make substantial progress on a number of problems in extremal combinatorics.

For each prime power \( q \) and integer \( 2 \le t \le q-1 \), we construct \( t \)-line evasive subsets of \( \mathbb{F}_q^n \) of size
\[
q^{\,n\left(1-\frac{2}{t^2+t}\right)},
\]
which is significantly larger than those previously known. Moreover, our method yields a partition of \( \mathbb{F}_q^n \) into such sets.

We extend this partitioning result to the projective space \( \PG(n,q) \), obtaining the first explicit colorings for the vector space Ramsey number \( R_q(2;k) \) that exhibit dependence on both \( q \) and \( k \). In particular, we show that
\[
R_q(2;k) > \frac{(q-1)k}{2} - O_q(1),
\]
improving recent bounds.

Finally, we apply these constructions to extremal graph theory and improve the best-known bounds on the bipartite Tur\'an number \( \mathrm{ex}(n,m,\{C_4,\theta_{3,t}\}) \). Most notably, we show that
\[
\mathrm{ex}(n,n^{2/3},\{C_4,\theta_{3,3}\}) = \Theta(n^{1+1/9}),
\]
making progress on a question originally posed by Erd\H{o}s.
\end{abstract}

\section{Introduction}

Let $q$ be a prime power and denote by $\F_q$ the finite field of order $q$. We denote the affine space of dimension $n$ over $\F_q$ by $\AG(n, q)$ and by $\F_q^n$ interchangeably. Denote the $n$-dimensional projective space over $\F_q$ by PG$(n, q)$.

A  $(t, s)$-set $K$ is a subset of points of PG$(n, q)$ such that any $s$-space of PG$(n, q)$ intersects $K$ in at most $t$ points. The study of such sets has historically been, and still is, a very active research area. Furthermore, large $(t, s)$-sets have important applications in coding theory and beyond \cite{HirschfeldStorme2001} for various sets of parameters $t$ and $s$. The following is a list of some traditionally well-studied parameters and the names associated with the corresponding subset of points.

\begin{enumerate}
    \item A $(2, 1)$-set $K$ is called a \textit{cap} in PG$(n, q)$, i.e., any line of PG$(n, q)$ intersects $K$ in at most 2 points.
    \item A $(t, t-1)$-set $K$ is referred to as a \textit{$t$-general set}, or a \textit{generalized cap}.
    \item An $(n, n-1)$-set in PG$(n, q)$ is called an \textit{arc}.
\end{enumerate}

More recently, there has been a significant amount of research into the analog of a $(t, s)$-set in the affine space AG$(n, q)$, called an $(s, t)$-subspace evasive set \cite{PudlakRodl2004}. To be precise, an $(s, t)$-subspace evasive set in AG$(n, q)$ is a set of points $K$ such that any $s$-dimensional flat of AG$(n, q)$ intersects $K$ in at most $t$ points. If $t = 1$, then $S$ is instead called $t$-line evasive. When $s = 2$ alongside $t=1$, then $S$ is called a cap set in AG$(n, q)$. Observe that since AG$(n, q)$ embeds into the projective space PG$(n, q)$ as a subgeometry, any $(s, t)$-subspace evasive set in AG$(n, q)$ may be embedded as a $(t, s)$-set in PG$(n, q)$. The usual line of inquiry regarding $(t, s)$-sets and $(s, t)$-subspace evasive sets is to determine their maximum size. In this paper, we are particularly interested in the case that $s = 1$. Denote by $\mpr{t}{n}{q}$ and $\mpa{t}{n}{q}$ the maximum size of a $(t, 1)$-set in $\PG(n, q)$ and $t$-line evasive set in $\AG(n, q)$ respectively.

One of the most notorious open problems in finite geometry is to determine the order of growth of the largest possible cap in PG$(n, q)$, in other words, determining the order of magnitude of $\mpr{2}{n}{q}$. The current best-known general bounds state \cite{HirschfeldStorme2001}
\begin{equation}
    q^{2n/3} + o(q^{2n/3}) <  \mpr{2}{n}{q} < q^{n-1} - o(q^{n-1})
\end{equation}
which is quite a large gap as $n$ increases. 

For certain values of $q$, this bound has been improved. For example, when $q = 3$, determining $\mpa{2}{n}{3}$ asks precisely for the largest subset of $\F_3^n$ which intersects each affine line in at most 2 points, and so we recover the legendary cap set problem.
After a series of improvements (see e.g. \cite{Edel}, \cite{Tyrrell}), the best known lower bound $\mpa{2}{n}{3} \gg 2.220^n$ was given by Romera-Paredes et al.~\cite{DeepMind}.
The first and only improvement to the exponential part of the trivial upper bound of $3^n$ is due to Ellenberg and Gijswijt \cite{EG}, who gave $\mpa{2}{n}{3} \ll2.756^n$.

In this paper, one of our primary concerns is the size of the  largest $t$-line evasive set in AG$(n, q)$ and hence the largest $(t, 1)$-set in $\PG(n, q)$. We are not the first to study this problem. Indeed, for a fixed $t$ and $q$, Lin and Wolf \cite{LW} proved that 
\begin{equation}
   \mpa{t}{n}{q}=  \Omega\left(q^{n\left(1 - \frac{1}{t+1}\right)}\right).
\end{equation}
In this paper, we generalize the construction of Lin and Wolf to significantly improve this lower bound significantly. 

\begin{thm}
\label{mainthm}
    Let $q$ be a prime power, and $t < q$. Then 
    $$
    \mpa{t}{n}{q} = \Omega\left(q^{n\left(1  - \frac{2}{t^2 + t}\right)}\right).
    $$
\end{thm}

Furthermore, we demonstrate that our constructions are maximal, and in fact, the proof of Theorem \ref{mainthm} gives a partition of $\F_q^n$ into maximal $t$-line evasive sets each of size $q^{n\left(1  - \frac{2}{t^2 + t}\right)}$ when $\binom{t+1}{2}|n$. Furthermore, we are able to extend our results to the projective space $\PG(n, q)$, which has a number of significant applications. Below, we survey each application and state our results.

\subsection{Vector Space Ramsey Numbers}

The vector space Ramsey number $R_q(s; k)$ is the smallest dimension $n$ such that any coloring of the 1-dimensional subspaces of $\F_q^n$ with $k$ colors yields a monochromatic $s$-space. Equivalently, it is the smallest integer $n$, such that any coloring of the points of $\PG(n-1, q)$ contains a monochromatic $(s-1)$-projective subspace. In particular, $R_q(2; k)$ is the smallest integer $n$ such that any $k$-coloring of the lines of $\PG(n-1, q)$ contains a monochromatic line. 

The study of the function $R_q(s; k)$ was first initiated by Graham, Leeb, and Rothschild \cite{GrahamLeebRothschild1972} where they proved that $R_q(s; k)$ is finite. Certain sets of parameters have received special attention, and some even overlap with historically significant problems in finite geometry. For example, $R_q(s; 2)$ can be translated as the smallest integer $n$, such that any blocking set for projective $(s - 1)$-spaces in $\PG(n-1, q)$ must contain an $(s-1)$-space. 

From here on, we will focus specifically on the Ramsey number $R_q(2; k)$. The best-known general upper bound for $R_q(2; k)$ to date is still very large, a tower function dependent on $q$ and $k$ \cite{FredericksonYepremyan2023}. For $q = 2$ or $3$, much stronger bounds are known. For $q = 2$, after a series of improvements, the most recent of which appears in \cite{BishnoiCamesRavi2025}, we have that 
$$
R_2(2; k) = O(k \log k).
$$
For $q = 3$, it is shown in \cite{FredericksonYepremyan2023} that
$$
R_3(2; k) = O(kC^k),
$$
where $C \approx 13.901.$

The lower bounds have also seen improvement very recently. In \cite{FredericksonYepremyan2025}, it is stated that a standard probabilistic argument yields that 
$$
R_q(2; k) = \Omega\left(\frac{q^2 \log_qk}{2}\right).
$$
This bound is improved upon by Bishnoi, Cames van Batenburg, and Ravi \cite{BishnoiCamesRavi2025} where the authors study the chromatic number of projective spaces, which has a direct relation with the Ramsey number $R_q(2; k)$. The chromatic number of the projective space $\PG(n-1, q)$, denoted $\chi_q(n)$, is the smallest number of colors required to color the points of PG$(n-1, q)$ so that no line is entirely contained in any one color class. From this definition, it follows that 
\begin{equation}\label{Equiv}
R_q(2;k) > n \iff   \chi_q(n) \leq k.
\end{equation}

The authors prove a lemma (\cite[Lemma 1]{BishnoiCamesRavi2025}) which establishes a recursive upper bound for $\chi_q(n)$, stating that for any $d < n$, 
\begin{equation}\label{Recursion}
    \chi_q(n) \leq \chi_q(n-d) + \chi_q(d).
\end{equation}
Using this recursion, they obtain the following result.
\begin{thm}{\cite[Theorems 1 and 3]{BishnoiCamesRavi2025}}
    For every integer $n \geq 2$, the following holds:
    \begin{itemize}
        \item When $q = 2$, 
    $$
    \chi_2(n) \leq \left\lfloor \frac{2n}{3}\right\rfloor + 1.
    $$
        \item When $q = 3$ or 4,
    $$
    \chi_q(n) \leq \frac{3n}{5} + O(1).
    $$
        \item For all fixed prime powers $q \geq  5$, 
    $$
    \chi_q(n) \leq \frac{n}{2} +O(1).
    $$
    \end{itemize}
\end{thm}

From (\ref{Equiv}) and the theorem above, the authors immediately obtain the following corollary.

\begin{cor}{\cite[Lemma 3]{BishnoiCamesRavi2025}}
    For every integer $k \geq 1$, the following holds: 
    \begin{itemize}
        \item If $q = 2$, 
        $$
        R_2(2; k) \geq \frac{3}{2}k - O(1).
        $$
        \item For $q = 3$ and 4,
    $$
    R_q(2; k) \geq \frac{5k}{3} - O(1).
    $$
        \item For all fixed prime powers $q \geq  5$, 
    $$
    R_q(2; k) \geq  2k  - O(1).
    $$
    \end{itemize}
\end{cor}

In \cite[Section 6]{BishnoiCamesRavi2025}, the authors pose the question of determining the asymptotic growth of $\chi_q(n)$ and suggest that a first step in this direction is to establish an upper bound which is dependent on $q$. We are able to establish precisely such a bound by providing an explicit coloring which proves that $\chi_q(\binom{q}{2} + 1) \leq q$ (see Section 4). In tandem with the established recursion formula (\ref{Recursion}), we obtain improved bounds on $\chi_q(n)$ and consequently on $R_q(2; k)$. Furthermore, in the regime where $n < \binom{q/2}{2}$, we obtain a stronger bound (which is particularly notable for large $q$).

\begin{thm}\label{ChromNum}
    For all fixed prime powers $q$,
    $$
    \chi_q(n) < \frac{2n}{q-1} + q.
    $$
    When $n \leq \binom{q/2}{2}$, we also have that 
    $$
    \chi_q(n) \leq \sqrt{8n} + 4.
    $$
\end{thm}

This immediately implies the following corollary:

\begin{cor}
    For all fixed prime powers $q$ and integers $k \geq q$,
    $$
    R_q(2; k) \geq \frac{q-1}{2}k - O_q(1).
    $$
    When $k \leq q$, we also have that 
    $$
    R_q(2; k) > \frac{(k - 4)^2}{8}.
    $$
\end{cor}

\subsection{Tur\'{a}n numbers}

Finally, we observe an application of our construction to a problem in extremal graph theory first posed by Erd\'{o}s \cite{Erdos1979}.

Let $m, n$ be positive integers and $\mathcal{F}$ be a family of graphs. The bipartite Tur\'{a}n number $ex(n, m, \mathcal{F})$, is the maximum number of edges in a bipartite graph whose part sizes are $m$ and $n$, such that it contains no graph in $\mathcal{F}$ as a subgraph. The function $ex(n, m, \mathcal{F})$ has been studied extensively for many different sets $\mathcal{F}$, but many questions remain. See the well-known survey by F\"{u}redi and Simonovits \cite{FurediSimonovits2013} for a history of the work done on these types of problems.

One of the most notorious cases of determining $ex(n, m, \mathcal{F})$ is when $\mathcal{F} = \{C_{2k}\}$ for some positive integer $k$. When $m = n$, the order of magnitude of $ex(n, n, \{C_{2k} \})$ is known for $k = 2, 3, 5$ \cite{FurediSimonovits2013}, which coincides with the existence of special finite geometries called generalized polygons. When $m = n^a$ with $a < 1$, and we allow $n \rightarrow \infty$, even less is known. The only well-understood case is that of  $ex(n, m, \{C_4\})$, where extremal constructions are precisely those arising from point-block incidence graphs of 2-designs \cite{FurediSimonovits2013}.

In 1979, Erd\H{o}s conjectured that when $m = O(n^{2/3})$, then 
$ex(m, n, \{C_4, C_6 \}) = O(n)$ \cite{Erdos1979}. De Caen and Sz\'{e}kely \cite{deCaenSzekely1997} proved that
$$
cn^{1 + \frac{1}{57} + o(1)} \leq ex(n, n^{2/3}, \{C_4, C_6\}) \leq n^{1 + 1/9} 
$$
for a constant $c$, thereby disproving Erd\H{o}s's conjecture.
Currently, the best-known lower bound for this problem is given by Lazebnik, Ustimenko and Woldar \cite{LazebnikUstimenkoWoldar1995} who constructed an infinite family of graphs which established 
$$
(1-o(1))n^{1 + \frac{1}{15}} \leq ex(n, n^{2/3}, \{C_4, C_6\}).
$$

A common generalization of the even cycle $C_{2k}$, is the graph $\theta_{k, t}$, which is the graph consisting of two vertices joined by $t$ vertex-disjoint paths of length $k$. Observe that the graph $\theta_{k, 2}$ is precisely the cycle $C_{2k}$. In this work, we improve both the upper and lower bounds for $ex(n, m, \{ C_4, \theta_{3, t}\})$.

The best-known upper bounds on $ex(n, m, \{C_4, \theta_{3, t} \})$ follow from Jiang, Ma,
and Yepremyan \cite{JiangMaYepremyan2022} who proved that when only $\theta_{3, t}$ is excluded, then
$$
ex(m, n, \{ \theta_{3, t} \}) \leq 144t^3((mn)^{\frac{2}{3}} + m + n).
$$
For the lower bounds, D\"{u}zg\"{u}n, Riet, and Taranchuk \cite{DuzgunRietTaranchuk2025} proved that $(t, 1)$-sets in PG$(n, q)$ give rise to bipartite graphs which are $C_4$-free and $\theta_{3, t}$-free. They used the $t$-line evasive sets constructed by Lin and Wolf \cite{LW}, alongside the upper bound of Jiang, Ma, and Yepremyan \cite{JiangMaYepremyan2022} to demonstrate that 
$$
ex(n, n^{2/3}, \{C_4, \theta_{3, 4}\}) = \Theta(n^{1 + 1/9}).
$$
More generally their work implies that
\begin{equation}\label{DRT}
ex(n, n^{a}, \{ C_4, \theta_{3, t} \}) = \Theta(n^{(2 + 2a)/3})    
\end{equation}
where  $a = \frac{j+1}{2j - 1}$ and $2 \leq j \leq t+1$.

The $t$-line evasive sets (and hence $(t, 1)$-sets) we construct in this paper are significantly larger than those given by Lin and Wolf \cite{LW}. This allows us to significantly expand the domain of the value $a$ for which an asymptotically tight bound holds for $ex(n, m, \{C_4, \theta_{3, t}\})$, as compared with the domain of $a$ in (\ref{DRT}). Furthermore, we also establish an upper bound on $ex(n, m, \{C_4, \theta_{3, t}\})$ which is stronger than the upper bound given in \cite{JiangMaYepremyan2022} where the authors only exclude $\theta_{3, t}$.

\begin{thm}\label{NoThetas}
    Let $j$ be an integer satisfying $2 \leq j \leq \binom{t + 1}{2}$ and $a = \frac{j+1}{2j - 1}$. Then
    $$
(1 - o(1))n^{(2 + 2a)/3} \leq ex(n, n^a, \{C_4, \theta_{3, t}\}) \leq (\sqrt[3]{t-1}+o(1))n^{(2 + 2a)/3}.
$$
\end{thm}

This theorem immediately implies the following corollary.

\begin{cor}\label{CorTheta}
    Let $n$ be a positive integer. Then
    $$
    ex(n, n^{2/3}, \{C_4, \theta_{3, 3} \}) = \Theta(n^{1 + 1/9})
    $$
\end{cor}

\subsection{Outline of paper}

The rest of the paper is organized as follows: In Section 2, we construct a polynomial $P_{t, q}$ in $\binom{t+1}{2}$ variables over $\F_q$ whose set of solutions to $P_{t, q} = u$ is $t$-line evasive for each $u \in \F_q$. Hence, we partition $\F_q^{\binom{t+1}{2}}$ into $q$, $t$-line evasive sets. In Section 3, we prove that each such set is in fact a maximal $t$-line evasive set over any subfield of $\F_q$. In Section 4, we demonstrate that the partitioning result can be extended to the projective space, giving improved bounds on the vector space Ramsey numbers.  In Section 5, we improve both the upper and lower bounds to the Tur\'{a}n number $ex(n, m, \{ C_4, \theta_{3, t}\})$. In Section 6, we mention some concluding remarks regarding applications to arithmetic progression-free sets and make some observations regarding the polynomial we construct in Section 2.

\section{Proof of Main Theorem}
The following proof is inspired by Lin and Wolf \cite[Proposition~3]{LW}. As in their proof, we use the existence of homogeneous polynomials of degree $d$ in $d$ variables over $\mathbb{F}_q$ with only the trivial root. In particular, when $d < q$, if we associate $\F_q^d$ with the field $\F_{q^d}$, then the norm function 
$$
N_d(x) = x^{q^{d-1} + \cdots + q + 1}
$$
yields precisely such a function when interpreted as a polynomial in $d$ variables over $\F_q^d$, see \cite[page 6]{borevich1986number}. 
It's explicit representation as a homogeneous $d$-variate polynomial can be determined by fixing a choice of basis for $\F_{q^d}$ over $\F_q$. 

We make one more observation before beginning the proof. Let $a, b \in \F_{q^d}$, $x \in \F_q$, and $N_d:\F_{q^d} \rightarrow \F_q$ be the norm function, then 
\begin{equation}\label{NormExpand}
N_d(ax + b) = N_d(a)x^d + tr(ba^{q + q^2 + \cdots +q^{d-1}})x^{d-1} + g_{a, b}(x)    
\end{equation}
where $tr$ is the trace functions from $\F_{q^d}$ to $\F_q$ and $g_{a, b}(x)$ is a polynomial of degree at most $d-2$. 

\begin{proof}[Proof of Theorem \ref{mainthm}]
Let $t \leq q-1$. Identify the vector space $\F_q^{\binom{t+1}{2}}$ with $V = \F_q \times \F_{q^2}\times \cdots \times\F_{q^{t}}$, and so to write $v = (v_1, v_2, \dots, v_{t}) \in V$ means $v_i \in \F_{q^i}$ for each $i$. Now consider the polynomial $P_{t, q}: \F_q^{\binom{t+1}{2}} \mapsto \F_q$ given by
$$
P_{t, q}(v) = N_{t}(v_{t}) + N_{t-1}(v_{t-1}) + \cdots + N_2(v_2) + v_1.
$$
Let $u \in \F_q$, and define 
$$
S_u = \{ v \in V: P_{t, q}(v) = u\}.
$$
\noindent \textbf{Claim}: For any affine line $\ell$ in $V$, $|\ell \cap S_u | \leq t$. 
Suppose for a contradiction that $|\ell \cap S_u| > t$. Parameterize $\ell = \{ ax + b: x\in \F_q \}$ where $a, b \in V$. Then, the polynomial $P_{t, q}(ax + b) - u$
in $x$ has at least $t+1$ roots. Note that by (\ref{NormExpand}), $P_{t, q}(ax + b) - u$ is of degree at most $t$, so if it has more than $t$ roots, then $P_{t, q}(ax + b)$ must be identically 0. Let $a = (a_1, a_2, \dots, a_{t}), b = (b_1, b_2, \dots, b_t)$, and $i \in [t] := \{ 1, 2, \dots, t \}$, be the largest index for which $a_i \neq 0$. This implies that
$$
P_{t, q}(ax + b) - u = N_t(a_tx + b_t) + N_{t-1}(a_{t-1}x + b_{t-1}) + \cdots N_2(a_2x + b_2) + a_1x + b_1-u
$$
has degree $i$ and the coefficient of $x^i$ is precisely $N_i(a_i) \neq 0$. However, if this was the case, it would contradict the fact that $P_{t, q}(ax + b) - u$ is identically 0. Thus $a_i =0$ for each $i$ and so $a = (0, \dots, 0)$, which implies that $\ell = \{ b\}$, a contradiction. Thus, $|\ell \cap S_u| \leq t$ and so $S_u$ is $t$-line evasive and may also be viewed as a $(t,1)$-set in $\PG(\binom{t + 1}{2}, q)$.

Furthermore, we note that $|S_u| = q^{\binom{t+1}{2}- 1}$ since for any choice of $v_2, v_2, \dots, v_{r}$, there exists a unique $v_1$ for which $P_{t, q}((v_1, \dots, v_{t})) = u$.

When $n > \binom{t+1}{2}$, let $k$ be the integer for which the inequality $\binom{t+1}{2}(k-1) <n \leq \binom{t+1}{2}k$ holds. The proof above implies that the set of solutions to $P_{t, q^k}(v) = u$ is $t$-line evasive in $\AG(\binom{t+1}{2}, q^k)$ for each $x \in \F_{q^k}$. Thus, we have a partition of the points of $\AG(\binom{t+1}{2}, q^k)$ into $q^k$ parts such that each part is $t$-line evasive. Performing field reduction yields a partition of $\AG(\binom{t+1}{2}k, q)$ into $q^k$ sets, each of which are $t$-line evasive over $\F_q$. Since $\AG(n, q)$ is a subspace of $\AG(\binom{t+1}{2}k, q)$, then the same holds true for $\AG(n, q)$. Hence, $\AG(n, q)$ contains a $t$-line evasive set of order at least $q^{n-k}$. Since $\binom{t+1}{2}(k-1) < n $, then $k < n/ \binom{t+1}{2} + 1$. Hence, 
$$
\mpa{t}{n}{q} \geq q^{n - k} > q^{n - n/\binom{t+1}{2} - 1} = q^{n\left( 1 - \frac{2}{t^2 + t} \right) - 1}.
$$
When $\binom{t+1}{2}$ divides $n$, then it follows that the set has order precisely $q^{n\left( 1 - \frac{2}{t^2 + t} \right)}$
\end{proof}

\section{Maximality of the line evasive sets from Theorem \ref{mainthm}}

In this section, we prove that our construction of $t$-line evasive sets in $\AG(\binom{t+1}{2}, q)$ is maximal over $\F_q$ as well as over any subfield of $\F_q$. This leads to a number of important consequences.

As in Section 2, let $V = \F_q \times \F_{q^2} \times \cdots \times\F_{q^t}$ and 
$$P_{t, q}(v) = N_{t}(v_{t}) + N_{t-1}(v_{t-1}) + \cdots + N_2(v_2) + v_1.$$

\begin{lem}
    Let $S_{x}$ be the set of roots of $P_{t, q}(v) - u$. Then for any $b \in V $ and any polynomial $f(x) \in \F_q[x]$ of degree at most $r$ with constant term $P_{t, q}(b) - u$, there exists a line $\ell = \{ax + b: u \in \F_q \}$ for which 
    $
    P_{t, q}(ax + b) - u = f(x).
    $
\end{lem}

\begin{proof}
Let $b \in V\setminus S_u$ be a fixed point and $a = (a_1, a_2, \dots, a_t)\in V$ be a variable slope. We will show that there is at least one slope $a$ for which $P_{t, q} - u$ restricted to the line $\ell = \{ax + b : x \in \F_q\}$ is equal to $f(x)$. Observe that it follows from (\ref{NormExpand}) that in the expansion of $P_{t, q}(ax + b)$,  the coefficient of $x^i$ is of the form 
    $$
    N_i(a_i) + g_i(a_{t}, \dots, a_{i + 1}, b_{t}, \dots, b_{i+1})
    $$
    where we emphasize that:
    \begin{enumerate}
        \item $g_i$ is independent of $a_i$.
        \item If $a_{i + 1} = \cdots =a_t = 0$, then $g_i = 0$.
        \item $N_i(a_i) = 0$ if and only if $a_i = 0$.
    \end{enumerate}
    Hence, we have that 
    \begin{equation}\label{Pexpand}
    P_{t, q}(ax + b)-u = N_t(a_t)x^t + \cdots + (N_i(a_i) + g_i(a_{t}, \dots, a_{i + 1}, b_{t}, \dots ,b_{i+1})x^i + \cdots + P_{t, q}(b) - u
    \end{equation}
    Now suppose $f(x) = c_dx^d + \cdots +c_1x + P_{t, q}(b) - u$. If $d < t$, set $c_t = \cdots = c_{d+1} = 0$. So we obtain the following system of equations 
    \begin{align*}
        N_t(a_t) &= c_t \\
        N_{t-1}(a_{t-1}) &= c_{t-1} - g_{t-1}(a_t, b_t) \\
        & \vdots  \\
        a_1 & = c_1 - g_1(a_t, \dots, a_2, b_t, \dots, b_2)
    \end{align*}
    which has at least one non-zero solution $(a_1, a_2, \dots, a_t)$ and usually many. Observe that $(a_1, a_2, \dots, a_t ) = (0, \dots, 0)$ is a solution if and only if $c_1 = c_2 \cdots  = c_t = 0$. Furthermore, the existence of a non-zero solution follows from the fact that $N_i(a_i)$ is surjective, as well as the fact that the system is triangular (i.e. the new variable which appears on the left-hand side in each line does not appear the equations preceding it). Hence, there is at least one slope $(a_1, \dots, a_t)$ for which $P_{t, q}(ax + b) - u = f(x)$.
\end{proof}

It is an immediate corollary that the set of roots $S_u$ of $P_{t, q}(v) - u$ is a maximal $t$-line evasive set not only in $AG(\binom{t+1}{2}, q)$, but also over any subfield of $\F_q$. Let $\F_{q'}$ be a subfield of $\F_q$, we call a set of points in $\AG(n, q)$ which form an affine line over $\F_{q}'$, an $\F_q'$-line.

\begin{thm}\label{maxii}
    Let $q$ and $q'$ be prime powers satisfying $q = (q')^k$, and $t$ be a positive integer satisfying $t < q'$. The set of roots $S_u$ of $P_{t, q}(v)-u$ form a maximal $t$-line evasive set in $\AG(k\binom{t+1}{2}, q')$.
\end{thm}

\begin{proof}
Let $b \in V\setminus S_u$, and therefore $P_{t, q}(b) - u \neq 0$. Choose $x_1, x_2, \dots, x_t \in \F_{q'}^*$ to be distinct elements and set $f(x) = c(x-x_1)\cdots(x - x_t)$ where $c \in \F_q$ satisfies $cx_1x_2\cdots x_t = P_{t, q}(b) - u$. By Lemma \ref{Pexpand}, there exists a line $\ell = \{ ax + b : x \in \F_q\}$ for which $P_{t, q}(ax + b) = f(x)$. Hence, there is an $\F_{q'}$-line in $\AG(\binom{t+1}{2}, q)$ which contains the point $b$ and $t$ points of $S_u$. So $S_u$ is a maximal $t$-line evasive set over $\F_{q'}$.
\end{proof}

\section{Improved Vector Space Ramsey Number Bounds}

It is an immediate corollary of the proof of Theorem \ref{mainthm} that $\AG(\binom{q}{2}, q)$ may be partitioned into $q$ parts, such that no part contains an affine line. Clearly this implies that for any $n \leq \binom{q}{2}$, there is a coloring of the points of $\AG(n, q)$ with at most $q$ colors such that no line in $\AG(n,q)$ is monochromatic.  The following lemma follows as a result.

\begin{lem}\label{lem:chromatic}
    Let $q$ be any prime power. Then 
    $$
    \chi_q\left(\binom{q}{2} + 1\right) \leq q \hs \text{ and } \hs R_q(2;q) \geq  \binom{q}{2} .
    $$
\end{lem}

\begin{proof}
    Let $n = \binom{q}{2}$. Color the points of $\AG(n, q)$ with $q$ colors in such a way that no affine line is monochromatic. Consider now the projective closure of $\AG(n, q)$, which adds an extra point on each line of $\AG(n, q)$, and this new set of points defines a hyperplane $\mathcal{H}$ of $\PG(n, q)$. Since the affine points satisfy the property that no affine line is monochromatic, then if we complete the coloring of $\PG(n, q)$ by coloring the points of $\mathcal{H}$, the only lines that can be monochromatic in the resultant coloring are those entirely contained in $\mathcal{H}$. In particular, the $q$-coloring of the affine part we begin with gives no restriction on how we may color the points of $\mathcal{H}$. We then repeat this coloring recursively: Take the affine part of $\mathcal{H}$, and color it with the same $q$ colors so that no affine line is entirely monochromatic. Such a coloring can easily be obtained by inducing the coloring of $\AG(\binom{q}{2}, q)$ onto a hyperplane in the affine space. As we mentioned, no line of $\PG(n, q)$ which is not entirely contained in $\mathcal{H}$ can be monochromatic. Likewise, in this coloring of affine part of $\mathcal{H}$, no affine line is monochromatic, so we are again free to color the points of the codimension-2 space as we wish. We repeat until we have colored all the points of $\mathcal{H}$. Thus, we have a coloring of $\PG(\binom{q}{2}, q)$ with $q$ colors in which no line is monochromatic. 
\end{proof}

\begin{proof}[Proof of Theorem \ref{ChromNum}]
When $n \leq \binom{q}{2}$, Lemma~\ref{lem:chromatic} implies that $\chi_q(n) \leq q$. Assume then that $n > \binom{q}{2}$. We apply the recursion formula (\ref{Recursion}) established in \cite{BishnoiCamesRavi2025}, which yields
$$
\chi_q(n) \leq \chi_q\left(\binom{q}{2} + 1\right)\left \lceil \frac{n+1}{\binom{q}{2} + 1} \right\rceil < q\left( \frac{n+1}{\binom{q}{2} + 1}  + 1\right) < \frac{2n}{q-1} + q
$$
where the last inequality follows from the fact that $\binom{q}{2} < n$.

Suppose now that $n \leq \binom{q/2}{2}$. Let $s$ be the integer satisfying $\binom{s-1}{2}< n \leq \binom{s}{2}$, and so $s \leq q/2$. Then, considering the proof of Theorem \ref{mainthm}, we may color all points of $\AG(n, q)$ using $q$ colors in such a way that any affine line contains at most $s-1$ points of any one color.
Applying the argument of Lemma~\ref{lem:chromatic}, it follows that we may color the points of $\PG(n,q)$ with $q$ colors so that any projective line contains at most $s$ points of any color.

 Let $j$ be the largest integer satisfying $\frac{q + 1}{s} > j$. Then, we may combine any $j$ color classes into one color class, where the new combined color class satisfies the property that any line of $\PG(n, q)$ intersects it in at most $sj < q+1$ points. Thus, we can partition the $q$ color classes into parts which each contain $j$ color classes (except for the last part, which may contain less than $j$ color classes). No line is contained entirely in any one color class, therefore
 $$
 \chi_q(n + 1) \leq \left \lceil \frac{q}{j} \right\rceil < \frac{q}{j} + 1.
 $$
 Since $j$ is the largest integer satisfying $j < \frac{q + 1}{s}$, then $\frac{q + 1}{s} - 1 \leq j$. Thus,
 $$
 \chi_q(n + 1) < \frac{q}{j} + 1 \leq \frac{qs}{q + 1 - s}+1.
 $$
 Observe that the inequality $s \leq q/2$ implies that $\frac{qs}{q+1 - s} \leq 2s$. Finally, from $\binom{s-1}{2} < n$, we gather that 
 $$
 t \leq \frac{\sqrt{8n + 1} + 3}{2}.
 $$
 Thus, we obtain
 $$
 \chi_q(n+1) \leq 2t + 1 \leq \sqrt{8(n + 1) - 7} + 4 < \sqrt{8(n + 1)} + 4.
 $$
 Recalling that $R_q(2;k) > n \iff   \chi_q(n) \leq k$, it follows that 
 $$
 R_q(2; k) \geq \frac{(k-4)^2}{8}.
 $$
\end{proof}

\section{Improved bounds for Tur\'{a}n numbers}

In this section, we briefly discuss an application of our construction to the Tur\'{a}n number $ex(n, n^{a}, \{C_4, \theta_{3, t}\})$. Recall that 
$ex(n, m, \mathcal{F})$ is the largest number of edges in a bipartite graph $\Gamma$ with part sizes $m$ and $n$ such that no graph in $\mathcal{F}$ is contained in $\Gamma$ as a subgraph. 

The best-known upper bounds on $ex(n, m, \{\theta_{3, t} \})$ follow from Jiang, Ma,
and Yepremyan~\cite{JiangMaYepremyan2022} who proved that 
$$
ex(m, n, \{ \theta_{3, t} \}) \leq 144t^3((mn)^{\frac{2}{3}} + m + n).
$$
To our knowledge, no upper bounds which improve upon the above bound have been given in the literature on the size of $\{ C_4, \theta_{3, t}\}$-free graphs. For $t = 2$, where $\theta_{3, 2} = C_6$, a similar bound indeed follows from \cite{deCaenSzekely1997, Neuwirth}. We adapt the proof of De Caen and Sz\'ekely \cite{deCaenSzekely1997} to obtain the following.

\begin{thm}
\label{thm:TuranUpper}
    Let $n\geq m$. Then 
    $$
    ex(n, m, \{ C_4, \theta_{3, t} \}) \leq \sqrt[3]{t-1}(mn)^{2/3} + (1+o(1))tn^{3/4}m^{1/2} + m\sqrt[3]{n}+n.
    $$
\end{thm}

Plugging in $m = n^a$, we obtain the following:

\begin{cor}
    Let $a\in (\frac{1}{2}, 1)$ be a fixed real number. Then 
    $$
    ex(n, n^a, \{ C_4, \theta_{3, t} \}) \leq (\sqrt[3]{t-1} + o(1))n^{(2 + 2a)/3}.
    $$
\end{cor}

\begin{proof}[Proof of Theorem~\ref{thm:TuranUpper}]

Let $G=(M,N,E)$ be a bipartite graph with parts of orders $m = n^a$ and $n$.  Suppose that $G$ is $\{ C_4, \theta_{3, t}\}$-free. We partition $M=M_1\dot\cup M_2$ such that $x\in M_1 \iff deg(x) \geq  d$, and we partition $E=E_1\dot\cup E_2$ accordingly. Note that $G_1=(M_1,N,E_1)$ and $G_2=(M_2,N,E_2)$ are again $\{ C_4, \theta_{3, t}\}$-free. We will bound the number of edges in each graph separately.
    
    Take any pair of vertices in $G_1$. Since $G_1$ is $C_{4}$-free and $\theta_{3, t}$-free, there are at most $(t-1)$ paths of length 3 between them, hence there are at most $(t-1)m_1n$ undirected 3-paths. Now, we lower bound the total number of three paths. For any edge $x \sim y$, with $x \in M_1$ and $y \in N$, there are at least $(d-1)(\deg(y) - 1)$ 3-paths through $x \sim y$. Hence, in total, there are at least
    $$
    \sum_{y \in N}(d-1)\deg(y)(\deg(y)-1)
    $$
    undirected 3-paths. Hence 
    \begin{equation}\label{inequal}
     \sum_{y \in N}(d-1)\deg(y)(\deg(y)-1) \leq (t-1)m_1n.
    \end{equation}
    Now, we observe that 
    $$
    \sum_{y\in N}\deg(y)=|E_1| \geq m_1d \implies \sum_{y \in N}\deg(y)^2 \geq \frac{|E_1|^2}{n}.
    $$
    Thus, from (\ref{inequal}), we have
    $$
    (d-1)\frac{|E_1|^2}{n} - (d-1)|E_1| \leq (t-1)m_1n \leq (t-1)\frac{n}{d}|E_1|\implies |E_1|\leq   (t-1)\frac{n^2}{d^2-d}+n.
    $$
    Let $d = \frac{n^{5/8}}{m^{1/4}}$, then 
    $$
    |E_1| \leq (t-1)n^{3/4}m^{1/2}\left(1+o(1)\right)+n.
    $$

   Now, we count $|E_2|$.  
    Let $B$ be the is $m \times n$ biadjacency matrix of $G_2$. Then, given $x\in M_2$ and $y\in N$, $\sigma(BB^TB)_{x, y}$ is the number of walks of length $3$ between $x$ and $y$. Since $G$ is $C_4$-free, and $\theta_{3, t}$-free, 
    $$
    (BB^TB)_{x, y} \leq \left\{ \begin{array}{cc}
        (t-1), & x \not\sim y \\
        \deg(x) + \deg(y) - 1, & x \sim y .
    \end{array} \right.
    $$
    Now using the following inequality due to Atkinson, Watterson, and Moran \cite{AtkinsonWattersonMoran1959}:
    \begin{equation}\label{Atkinson}
        \frac{\sigma(B)^3}{mn} \leq \sigma(BB^TB),
    \end{equation} 
    we get
    \begin{align*}
    \frac{|E_2|^3}{m_2n}=\frac{\sigma(B)^3}{m_2n}\leq\sigma(BB^TB) \leq& (t-1)(m_2n - |E_2|) + \sum_{x \in M_2}\deg(x)^2+ \sum_{y \in N}\deg(y)^2-|E_2|. 
    \end{align*}

     From the above, we assume that $x \in M_2 \implies \deg(x) < \frac{n^{5/8}}{m^{1/4}}$. Hence, $\sum_{x \in M_2} \deg(x)^2 < \frac{n^{5/4}m_2}{m^{1/2}}$. On the other hand, note that since $G_2$ is $C_4$-free, then $\sum_{y \in N}\binom{\deg(y)}{2} \leq \binom{m_2}{2}$. Thus,
    \begin{align*}
        \sum_{x \in M_2}\deg(x)^2+ \sum_{y \in N}\deg(y)^2-|E_2| \leq \frac{n^{5/4}m_2}{m^{1/2}} + m_2^2 - m_2.
    \end{align*}
    It follows that  
    $$
    |E| = |E_1| + |E_2| < \sqrt[3]{t-1}(mn)^{2/3} + (1+o(1))tn^{3/4}m^{1/2} + m\sqrt[3]{n}+n.
    $$    
\end{proof}

In \cite{DuzgunRietTaranchuk2025}, the authors observe that any $(t, 1)$-set in $\PG(n, q)$ can be used to a construct a $\{C_4, \theta_{3, t} \}$-free graph via a geometry known as a linear representation of a point set. For completeness, we include the definition here.

\begin{defn}
    Let $S$ be a set of points of $\emph{PG}(n, q)$ and embed $\emph{PG}(n, q)$ as a hyperplane into $\emph{PG}(n+1, q)$. A \emph{linear representation} of $S$ is the geometry whose points are all the points in $\emph{PG}(n+1, q)\setminus \emph{PG}(n, q)$ and the lines are all the lines of $\emph{PG}(n+1, q)$ which intersect $\emph{PG}(n, q)$ in precisely one point, namely a point of $S$. 
\end{defn}

It is not hard to see that given a set of points $S$ in $\PG(n, q)$, the point-line incidence graph of the linear representation of $S$ is a graph with:
\begin{enumerate}
    \item $q^{n+1}$ point vertices, each of degree $|S|$. 
    \item $|S|q^n$ line vertices, each of degree $q$.
\end{enumerate}

 We now show that the upper bound we gave is asymptotically tight for many choices of $a \in [\frac{1}{2} + \epsilon_t, 1)$ where $\epsilon_t = \frac{3}{2t^2 + 2t - 2}$. In particular, we highlight that our construction of line evasive sets allows us to expand the domain of $a$ for which we may attain asymptotically tight bounds. This can be compared with the domain of $a$, which follows from using the line evasive sets of Lin and Wolf. For their sets it can be seen that $a \in [\frac{1}{2} + \epsilon_t, 1)$, where $\epsilon_t = \frac{1}{t+1}$.

\begin{proof}[Proof of Theorem \ref{NoThetas}]
    The upper bound has already been established, so we are left to prove the lower bound. 
    We have shown that $\AG(\binom{t+1}{2}, q)$ can be partitioned into $q$ parts, each satisfying the property that any line intersects each part in at most $t$ points. Consequently, for any $2 \leq j \leq \binom{t+1}{2}$, since $\AG(j, q) \subset \AG(\binom{t+1}{2}, q)$, then $\AG(j, q)$ contains a subset of at least $q^{j-1}$ points which satisfies the same property. We then embed $\AG(j, q)$ into $\PG(j, q)$ and hence obtain a $(t, 1)$-set in $PG(j, q)$ of the same size. As shown in \cite{DuzgunRietTaranchuk2025}, the linear representation of such a set in $\PG(j, q)$ is a bipartite graph with $q^{2j - 1}$ line vertices and  $q^{j+1}$ point vertices, which is $C_4$-free and $\theta_{3, t}$-free. By the fact that for all sufficient large $n$, there exists a prime between $n$ and $n^{0.525}$ \cite{BakerHarmanPintz2001}, the result follows for all sufficiently large $n$. 
\end{proof}

In \cite{DuzgunRietTaranchuk2025}, the authors proved that $ex(n, n^{2/3}, \{C_4, \theta_{3, 4}\}) = \Theta(n^{1 + 1/9})$. As an immediate corollary of Theorem \ref{NoThetas}, we improve upon this. Indeed, plug in $t = 3$ and $j = 5$ into Theorem \ref{NoThetas} to obtain a proof of Corollary \ref{CorTheta}, implying
$$
ex(n, n^{2/3}, \{C_4, \theta_{3, 3}\}) = \Theta(n^{1 + 1/9}).
$$

\section{Concluding Remarks}

Let $p$ be a prime. Since every $t$-term progression $\{a+ib\;|\;i\in[0,t-1]\}$ in $\mathbb{F}_p^n$ is contained in the line $\{a+ib\;|\;i\in\mathbb{F}_p\}$, $\mpa{t-1}{n}{p}$ can also be used as an upper bound for the maximal cardinality of a subset $S\subseteq \mathbb{F}_p^n$ not containing $t$ points in arithmetic progression, denoted by $r_t(\mathbb{F}_p^n)$.

 In a recent work Elsholtz, Hunter, Proske and Sauermann~\cite{EHPS} gave the lower bound $r_3(\mathbb{F}_p^n)\geq (cp)^n$ for some $c>\frac{1}{2}$. This was used to significantly improve the long-standing best lower bounds for progression-free sets over the integers due to Behrend~\cite{Behrend}.
 
For $t=p$ being a prime, we actually have $r_p(\mathbb{F}_p^n) = \mpa{p-1}{n}{p}$. In this case a hypercube of side length $p-1$ gives a trivial bound $r_p(\mathbb{F}_p^{n}) \geq (p-1)^n.$ Frankl, Graham and R\"odl \cite{FGR} gave the bound $r_p(\mathbb{F}_p^{2p}) \geq (p-1)^{2p-1}p.$ In \cite{EFFKPGV} several bounds for lower dimensions are discussed.

 Every $t$-progression-free set $S\subseteq\mathbb{F}_p^{n_1}$ provides an asymptotic lower bound $(|S|^{1/{n_1}}-o(1))^n$ for $r_t(\mathbb{F}_p^n)$ using the product construction $S\times S\times ... \times S$ (see e.g. \cite{Pach22}). The previously best asymptotic bound for sets containing no full line in $\mathbb{F}_p^n$ is derived from the result of Frankl, Graham and R\"odl \cite{FGR}. Theorem \ref{mainthm} gives an improved bound.

\begin{cor}
$$r_t\left(\mathbb{F}_p^{\binom{t}{2}}\right)\geq p^{\binom{t}{2}-1},$$
and in particular, 
$$r_p\left(\mathbb{F}_p^{n}\right)=\left(p+o_p(1)+o_n(1)\right)^n,$$ where $o_p(1)=o(1)$ as a function in $p$ and $o_n(1)=o(1)$ as a function in $n$.
\end{cor}

It is not hard to see that the polynomial we use to construct a $t$-line evasive set in $\AG(\binom{t+1}{2}, q)$ is of minimal degree. Indeed, it follows from Chevalley-Warning \cite{Chevalley1935, Warning1936} that any polynomial in $n$ variables and of degree $t$ for which $n > \binom{t+1}{2}$ must contain a full affine line in the set of roots of $P$. This can be observed by writing $P$ as the sum $P_t + P_{t-1} + \cdots +P_1 + P_0$ where $P_i$ is homogenous of degree $i$ in $n$ variables. Then Chevalley-Warning states that the system
\begin{eqnarray*}
    P_1 &= 0 \\
    P_2 &= 0 \\
    &\vdots \\
    P_t &= 0
\end{eqnarray*}
has a common non-zero solution. Suppose first that $P_0 = 0$. Since the polynomials $P_i$ are homogenous, then in fact all scalar multiples of the non-zero solution are also solutions. If $P$ has a non-zero constant term, and a root at $v$, then we consider the translated polynomial $P(x - v)$, whose set of roots is just an additive translation of the roots of $P$. Thus, $P$ contains a full line in its set of roots if and only if $P(x - v)$ does. Our argument above implies that this is indeed the case. It would be quite interesting to find other ``nice" polynomials which yield larger $t$-line evasive sets. 

\subsection*{Acknowledgements}

J.F. was supported by the Austrian Science Fund (FWF) under the projects W1230 and PAT2559123. The authors also thank Anurag Bishnoi, Christian Elsholtz and Leo Storme for their helpful comments.

\begingroup
\setlength{\emergencystretch}{2em}
\printbibliography
\endgroup

\end{document}